\documentclass[11pt]{amsart}
\usepackage{lmodern}
\usepackage{amsmath, amsthm, amssymb, amsfonts}
\usepackage[normalem]{ulem}
\usepackage{hyperref}

\usepackage{mathrsfs}

\usepackage{verbatim} 
\usepackage{longtable}

\usepackage{mathtools}

\usepackage{tikz}
\usetikzlibrary{decorations.pathmorphing}
\tikzset{snake it/.style={decorate, decoration=snake}}

\usepackage{caption}

\usepackage{tikz-cd}
\usetikzlibrary{arrows}

\theoremstyle{plain}
\newtheorem{thm}{Theorem}[section]

\newtheorem{lem}[thm]{Lemma}
\newtheorem{prop}[thm]{Proposition}

\theoremstyle{definition}

\theoremstyle{remark}
\newtheorem{rmk}[thm]{Remark}

\newcommand{\BA}{{\mathbb{A}}}

\newcommand{\BC}{{\mathbb{C}}}

\newcommand{\BF}{{\mathbb{F}}}

\newcommand{\BP}{{\mathbb{P}}}
\newcommand{\BQ}{{\mathbb{Q}}}

\newcommand{\BZ}{{\mathbb{Z}}}

\newcommand{\CA}{{\mathcal A}}
\newcommand{\CB}{{\mathcal B}}

\newcommand{\CD}{{\mathcal D}}
\newcommand{\CE}{{\mathcal E}}
\newcommand{\CF}{{\mathcal F}}

\newcommand{\CR}{{\mathcal R}}

\newcommand{\CU}{{\mathcal U}}

\newcommand{\FM}{{\mathfrak{M}}}

\DeclareFontFamily{OT1}{rsfs}{}
\DeclareFontShape{OT1}{rsfs}{n}{it}{<-> rsfs10}{}
\DeclareMathAlphabet{\curly}{OT1}{rsfs}{n}{it}
\renewcommand\hom{\curly H\!om}

\usepackage{tikz}
\usepackage{lmodern}
\usetikzlibrary{decorations.pathmorphing}

\begin{document}
\title[The period--index problem]{The period--index problem for hyper-K\"ahler varieties via hyperholomorphic bundles}
\date{\today}

\author[J. Hotchkiss]{James Hotchkiss}
\address{Columbia University}
\email{james.hotchkiss@columbia.edu}

\author[D. Maulik]{Davesh Maulik}
\address{Massachusetts Institute of Technology}
\email{maulik@mit.edu}

\author[J. Shen]{Junliang Shen}
\address{Yale University}
\email{junliang.shen@yale.edu}

\author[Q. Yin]{Qizheng Yin}
\address{Peking University}
\email{qizheng@math.pku.edu.cn}

\author[R. Zhang]{Ruxuan Zhang}
\address{Fudan University}
\email{rxzhang18@fudan.edu.cn}

\begin{abstract}
We prove new bounds for the period--index problem for hyper-K\"ahler varieties of $K3^{[n]}$-type using projectively hyperholomorphic bundles constructed by Markman. We show that $\mathrm{dim}(X)$ is a bound for any $X$ of $K3^{[n]}$-type. We also show that $\frac{1}{2}\mathrm{dim}(X)$ is a bound for most Brauer classes when the Picard rank of $X$ is at least two, providing evidence for a conjecture of Huybrechts.
\end{abstract}

\maketitle

\setcounter{tocdepth}{1} 

\tableofcontents
\setcounter{section}{-1}

\section{Introduction}

Throughout, we work over the complex numbers $\BC$.

\subsection{The period--index problem}

Let $X$ be a nonsingular projective variety,
and let $\alpha \in \mathrm{Br}(X)$ be a class in the Brauer group of $X$.
Two basic invariants of $\alpha$ are its \emph{period} and its \emph{index}.  
The period $\mathrm{per}(\alpha)$ is nothing more than the order of $\alpha$ in $\mathrm{Br}(X)$, which is a torsion group, 
while the index is given by $\mathrm{ind}(\alpha) = \gcd\{\sqrt{\mathrm{rk}(\mathcal{A})}\}$, as $\mathcal{A}$ runs over the Azumaya algebras of class $\alpha$.
It is an elementary fact that $\mathrm{per}(\alpha)$ and $\mathrm{ind}(\alpha)$ share the same prime factors, and that $\mathrm{per}(\alpha)$ divides $\mathrm{ind}(\alpha)$. 

In this paper, we are concerned with the \emph{period--index problem} for the variety $X$, which is the problem of determining the least positive integer $e(X)$ such that the inequality
\begin{equation}
\label{eq:period-index-problem}
    \mathrm{ind}(\alpha) \mid \mathrm{per}(\alpha)^{e(X)}
\end{equation}
holds for all $\alpha \in \mathrm{Br}(X)$. 
The motivation for the problem comes from the theory of central simple algebras over the function field $\mathbb{C}(X)$:
There is a natural inclusion $\mathrm{Br}(X) \subset \mathrm{Br}(\mathbb{C}(X))$, and the index of a class $\alpha \in \mathrm{Br}(X)$ coincides with $\deg(D) := \sqrt{\mathrm{dim}_{\mathbb{C}(X)}(D)}$, where $D$ is the unique division algebra over $\mathbb{C}(X)$ of class $\alpha$, viewed as an element of $\mathrm{Br}(\mathbb{C}(X))$.
Determining the degree of a central division algebra over a field in terms of its period is a fundamental classical problem, originating in the study of Brauer groups of local and global fields from the early part of the twentieth century; for an introduction, see \cite[Section 4]{Auel} or \cite{CT-bourbaki}.

If $X$ is a surface, then $\mathrm{per}(\alpha) = \mathrm{ind}(\alpha)$ for all $\alpha \in \mathrm{Br}(X)$ by a theorem of de Jong \cite{dJ}.
In higher dimensions, however, little is known, beyond the fact (proved in \cite{HM2}) that a bound as in \eqref{eq:period-index-problem} exists. 
From the longstanding period--index conjecture for the function field $\mathbb{C}(X)$ from~\cite{CT-german}, one expects
\begin{equation}
\label{eq:period-index-conj}
    \mathrm{ind}(\alpha) \mid \mathrm{per}(\alpha)^{\mathrm{dim}(X) - 1}
\end{equation}
for Brauer classes on all nonsingular projective varieties $X$; moreover, there are varieties of each dimension for which the bound \eqref{eq:period-index-conj} is sharp \cite{CT-gabber}.
In \cite{HP}, the period--index conjecture (\emph{i.e.},~$\mathrm{ind}(\alpha) \mid \mathrm{per}(\alpha)^2$) was proven for Brauer classes on abelian threefolds.

In \cite{H_PI}, Huybrechts studied the period--index problem for hyper-K\"ahler varieties, and proved the following result concerning Lagrangian fibrations. It shows that if $X$ admits a Lagrangian fibration, then $e(X)$ can be far smaller than the exponent $\mathrm{dim}(X)-1$ occuring in the period--index conjecture, at least if $\mathrm{per}(\alpha)$ avoids finitely many primes.

\begin{thm}[Huybrechts \cite{H_PI}]\label{thm0.1}
Let $X$ be a hyper-K\"ahler variety which admits a Lagrangian fibration. Then there exists an integer $N_X$ such that
\begin{equation}\label{H_Conj}
\mathrm{ind}(\alpha) \mid \mathrm{per}(\alpha)^{\frac{\mathrm{dim}(X)}{2}}
\end{equation}
for all $\alpha \in \mathrm{Br}(X)$ with $\mathrm{per}(\alpha)$ coprime to $N_X$.
\end{thm}

An explicit description of $N_X$ was given in \cite[(0.1)]{H_PI}. Huybrechts further conjectured that the equation (\ref{H_Conj}) holds with $N_X =1$ for any hyper-K\"ahler variety $X$ \cite[Conjecture 0.2]{H_PI}, and verified it for the Hilbert scheme of points on a $K3$ surface \cite[Theorem 0.4]{H_PI}.

The purpose of this note is to study the period--index problem for hyper-K\"ahler varieties of $K3^{[n]}$-type. 
Using Markman's projectively hyperholomorphic bundles from \cite{Markman}, we are able to obtain new bounds for the period--index problem. For varieties of $K3^{[n]}$-type, our results also recover and provide a new proof of Theorem \ref{thm0.1}.
The idea of applying projectively hyperholomorphic bundles to classical problems about Brauer groups goes back to the work of Huybrechts and Schr\"oer \cite{HS} on Grothendieck's question for Brauer groups of analytic $K3$ surfaces. 

\subsection{Hyper-K\"ahler varieties of $K3^{[n]}$-type}\label{sect0.2}

Throughout Section \ref{sect0.2}, we consider $X$ a hyper-K\"ahler variety of $K3^{[n]}$-type, \emph{i.e.}, it is deformation equivalent to the Hilbert scheme of~$n$ points on a $K3$ surface. As the $K3$ surface case has been dealt with, we assume $n \geq 2$.

Our first result provides a uniform bound for any $X$ of $K3^{[n]}$-type. For notational convenience, we set
\[
I_X:= \mathrm{Index}\left( T(X) \hookrightarrow H^2(X,Z)/\mathrm{Pic}(X) \right)
\]
where $T(X) \subset H^2(X, \BZ)$ is the transcendental lattice. Note that the embedding of the lattices on the right-hand side is always of finite index.

\begin{thm}\label{thm0.2}
Assume that $X$ admits a primitive polarization of degree $2h$. We have 
    \begin{equation}\label{weak_bound}
        \mathrm{ind}(\alpha) \mid \mathrm{per}(\alpha)^{\mathrm{dim}(X)}
\end{equation}
for all $\alpha \in \mathrm{Br}(X)$ with $\mathrm{per}(\alpha)$ coprime to $n!hI_X$.  
\end{thm}

In Section \ref{NumOb}, we discuss numerical obstructions to proving Huybrechts' conjecture (\ref{H_Conj}) using our method. In particular, our approach is not expected to be sufficient to prove Huybrechts' conjecture for the Picard rank $1$ case.

\begin{rmk}
To the best of our knowledge, (\ref{weak_bound}) gives the best bound currently known for $K3^{[n]}$-type varieties of Picard rank $1$. In \cite[Section 2.6]{H_PI}, Huybrechts proved that the Fano variety of lines $F(Y)$ of a nonsingular cubic $4$-fold $Y$, which is a variety of $K3^{[2]}$-type, satisfies
\[
\mathrm{ind}(\alpha) \mid \mathrm{per}(\alpha)^5, \quad \alpha \in \mathrm{Br}(F(Y))
\]
if $\mathrm{per}(\alpha)$ is coprime to some uniform integer depending on $Y$. In this case, Theorem \ref{thm0.2} reads 
\[
\mathrm{ind}(\alpha) \mid \mathrm{per}(\alpha)^4
\]
when $\mathrm{per}(\alpha)$ avoids finitely many prime factors.
\end{rmk}

When $X$ has higher Picard rank, much better bounds can be obtained. For further discussions, we need to introduce some notation. Recall the explicit description of the Brauer~group
\begin{equation*}
\mathrm{Br}(X) = \left( \frac{H^2(X, \BZ)} {\mathrm{Pic}(X)}\right) \otimes \BQ/\BZ,
\end{equation*}
which allows us to present a Brauer class in the form of a ``$B$-field'':
\begin{equation}\label{Brauer}
\alpha = \left[ \frac{\CB}{\ell}\right].
\end{equation}
Here $\ell$ is a positive integer, and $\CB$ is an element of the transcendental lattice $T(X) \subset H^2(X, \BZ)$ which is not divisible by any prime factor of $\ell$. We say that the Brauer class (\ref{Brauer}) is \emph{non-special}~if 
\[
\mathrm{gcd}(q(\CB), \ell) =1,
\]
where $q(\CB) \in \BZ$ is the norm with respect to the Beauville--Bogomolov--Fujiki (BBF) form $q(-, -)$ on~$H^2(X, \BZ)$.  If $\mathrm{per}(\alpha)$ is coprime to $I_X$, we further have
\[
\mathrm{per}(\alpha) =\ell.
\]

The following result is our main tool of constructing Brauer classes satisfying (\ref{H_Conj}).

\begin{thm}\label{thm0.4}
For a Brauer class $\alpha = \left[\frac{\CB}{\ell} \right]$ with $\mathrm{per(\alpha)}$ coprime to $n!I_X$, if there exist classes $\CD, \CD' \in \mathrm{Pic}(X)$ and a positive integer $d$ with $\mathrm{gcd}(d,\ell)=1$ satisfying that
\begin{enumerate}
    \item[(i)] $q(\CD) \equiv -q(\CB)d^2 \mod \ell$,
    \item[(ii)] $\mathrm{gcd}\left(q(\CD,\CD'), q(\CB), \ell \right)=1$,
\end{enumerate}
then (\ref{H_Conj}) holds for $\alpha$.
\end{thm}

As an application, Theorem \ref{thm0.4} implies immediately Theorem \ref{thm0.1} for hyper-K\"ahler varieties of $K3^{[n]}$-type; see Section \ref{revisit}. 

The next theorem is also a consequence of Theorem \ref{thm0.4}, which verifies (\ref{H_Conj}) for most Brauer classes when the Picard rank is at least $2$.

\begin{thm}\label{thm0.6}
    If $X$ has Picard rank at least $2$, then there exists an integer $N_X$ such that 
    \[
\mathrm{ind}(\alpha) \mid \mathrm{per}(\alpha)^{\frac{\mathrm{dim}(X)}{2}}    
\]
for all non-special $\alpha \in \mathrm{Br}(X)$ with $\mathrm{per}(\alpha)$ coprime to $N_X$.
\end{thm}

Here, the constant $N_X$ is explicitly determined by $I_X$ and the BBF form on $\mathrm{Pic}(X)$.

\subsection{Acknowledgements}
We are grateful to Daniel Huybrechts and Alex Perry for their comments on the first draft of the paper. J.S.~would like to thank Haoyi Shen for his kind cooperation during the preparation of this paper.

J.H.~was supported by the NSF grant DMS-2401818. D.M.~was supported by a Simons Investigator Grant. J.S.~was supported by the NSF grant DMS-2301474 and a Sloan Research Fellowship. R.Z.~was supported by the NKRD Program of China No.~2020YFA0713200 and LNMS.

\section{Projectively hyperholomorphic bundles}

Throughout this section, we assume that $X$ is a hyper-K\"ahler variety of $K3^{[n]}$-type with $n \geq 2$, and $\alpha = \left[\frac{\CB}{\ell}\right]$ is a Brauer class with $\mathrm{per}(\alpha)$ coprime to $I_X$; in particular we identify $\ell$ and $\mathrm{per}(\alpha)$. The main result of this section is Proposition \ref{prop1.3}.

\subsection{Brauer classes}

The period--index problem for $X$ can be reduced to a geometric question concerning the existence of certain twisted vector bundles.
A comparison of different notions of the index (e.g., the fact that the index of $\alpha$ viewed in $\mathrm{Br}(X)$ coincides with the index of $\alpha$ viewed in $\mathrm{Br}(\mathbb{C}(X))$) may be found in \cite{dJP}.

We write the Brauer class $\alpha = \left[\frac{\CB}{\ell} \right]$ on $X$ as in (\ref{Brauer}), and consider $\alpha$-twisted vector bundles on~$X$. The index of $\alpha$ is bounded above by the rank of any $\alpha$-twisted vector bundle $\CF$ on $X$:
\begin{equation}\label{rank}
\mathrm{ind}(\alpha) \mid \mathrm{rk}(\CF).
\end{equation}
Indeed, $\hom(\CF, \CF)$ is an Azumaya algebra on $X$ of class $\alpha$.
The following lemma is an immediate consequence.

\begin{lem}\label{lem1.1}
Assume $\alpha = \left[\frac{\CB}{\ell} \right]$ with $\mathrm{per}(\alpha) = \ell$. If there exists an $\alpha$-twisted vector bundle whose rank divides $c \ell^e$ for positive integers $c,e$ with $\mathrm{gcd}(c,\ell)=1$, then we have
\[
\mathrm{ind}(\alpha) \mid \mathrm{per}(\alpha)^e.
\]
\end{lem}

\begin{proof}
By (\ref{rank}) we have
\[
\mathrm{ind}(\alpha) \mid c\ell^e.
\]
Since $\mathrm{ind}(\alpha)$ and $\ell = \mathrm{per}(\alpha)$ share the same prime factors, the assumption that $c, \ell$ are coprime implies 
\[
\mathrm{gcd}(c, \mathrm{ind}(\alpha)) = 1.
\]
Hence we must have
\[
\mathrm{ind}(\alpha) \mid \ell^e = \mathrm{per}(\alpha)^e. \qedhere
\]
\end{proof}

\subsection{Projectively hyperholomorphic bundles} Markman constructed in \cite{Markman} a class of projectively hyperholomorphic bundles. They are our main sources of twisted vector bundles in view of Lemma \ref{lem1.1}. We briefly review the construction as follows, and refer to \cite{Markman} and \cite{MSYZ} for details.

Consider a projective $K3$ surface $S$ with $\mathrm{Pic}(S)=\BZ H$, and a primitive and isotropic Mukai vector
\[
v_0:=(r, mH, s) \in H^*(S, \BZ),\quad  v_0^2=0
\]
with
\[
m\in \BZ, \quad \mathrm{gcd}(r,s)=1,
\]
satisfying
\begin{equation}\label{VB}
    \frac{r}{\rho} \nmid \frac{1}{2}\left( \frac{mH}{\rho}\right)^2 +1, \quad \rho := \mathrm{gcd}(r,m).
\end{equation}
Clearly (\ref{VB}) guarantees that $r\geq 2$. By \cite[Lemma 1.2]{Y}, the numerical condition (\ref{VB}) implies that the moduli space $M$ of stable vector bundles on $S$ with Mukai vector~$v_0$ is again a $K3$ surface. Furthermore, the coprime condition on $r, s$ ensures the existence of a universal rank~$r$ bundle $\CU$ on $M \times S$. Conjugating the Bridgeland--King--Reid (BKR) correspondence yields a vector bundle $\CU^{[n]}$ on~$M^{[n]} \times S^{[n]}$ of rank
\[
\mathrm{rk}(\CU^{[n]}) = n!r^n;
\]
see \cite[Lemma 7.1]{Markman}. Markman further showed in \cite[Section 5.6]{Markman} that the characteristic class of $\CU^{[n]}$ induces a Hodge isometry
\begin{equation}\label{isometry}
    \phi_{\CU^{[n]}}: H^2(M^{[n]}, \BQ) \to H^2(S^{[n]}, \BQ). 
\end{equation}
Let $\Lambda$ be the $K3^{[n]}$-lattice; the Hodge isometry above induces an abstract isometry $\phi \in O(\Lambda_\BQ)$ on the lattice $\Lambda_\BQ$ up to equivalences.

We now consider the moduli space $\FM_{\phi}$ of isomorphism classes of quadruples $(Y, \eta_Y, X, \eta_X)$ where $(Y, \eta_Y), (X, \eta_X)$ are marked hyper-K\"ahler manifolds of $K3^{[n]}$-type and 
\[
\eta_X^{-1}\circ \phi \circ \eta_Y: H^2(Y, \BQ) \to H^2(X,\BQ)
\]
is a Hodge isometry sending some K\"ahler class of $Y$ to a K\"ahler class of $X$. Let $\FM^\circ_{\phi} \subset \FM_{\phi}$ be the connected component which contains the quadruple 
\[
(M^{[n]}, \eta_{M^{[n]}}, S^{[n]}, \eta_{S^{[n]}})
\]
canonically induced by (\ref{isometry}). A key result proved in \cite{Markman}, which is crucial for our purpose, is that the bundle $\CU^{[n]}$ is a projectively hyperholomorphic bundle which deforms along diagonal twistor paths in $\FM^0_\phi$. In particular, for every point 
\[
(Y, \eta_Y, X, \eta_X) \in \FM^0_{\phi},
\]
there exists a twisted vector bundle 
\[
\CE \rightsquigarrow Y\times X
\]
with respect to the Brauer class
\[
\alpha_\CE: = \left[- \frac{c_1(\CU^{[n]})}{n!r^n} \right].
\]
Here we view $H^2(Y\times X, \BZ) = H^2(Y, \BZ) \oplus H^2(X, \BZ)$ as a trivial local system over the moduli space $\FM^0_{\phi}$ via the markings, and therefore $c_1(\CU^{[n]})$ is a well-defined class over $Y\times X$:
\[
c_1(\CU^{[n]}) =  \alpha_Y + \alpha_X \in H^2(Y, \BZ) \oplus H^2(X,\BZ), \quad \alpha_Y \in H^2(Y, \BZ),\quad \alpha_X \in H^2(X, \BZ).
\]
Further restricting $\CE$ over a point $y \in Y$, we obtain a $\left[-\frac{\alpha_X}{n!r^n}\right]$-twisted vector bundle of rank
\[
\mathrm{rk}\left( \CE_y \right) = n!r^n.
\]
Note that under the canonical identification
\[
H^2(S^{[n]}, \BZ) = H^2(S, \BZ) \oplus \BZ \delta, \quad \delta^2 = 2-2n,
\]
the explicit formula of $\alpha_{S^{[n]}}$ over the Hilbert scheme $S^{[n]}$ is given by 
\[
\alpha_{S^{[n]}} = n!r^n \left(\frac{mH}{r} - \frac{\delta}{2} \right);
\]
see \cite[Equation (7.11)]{Markman}. 

We now summarize the discussion above in the following proposition.

\begin{prop}\label{prop1.2} Let $X$ be a variety of $K3^{[n]}$-type with a Brauer class $\alpha = \left[ \frac{\CB}{\ell} \right]$. Assume that there exist 
\begin{enumerate}
    \item[(i)] a projective $K3$ surface with $\mathrm{Pic}(S)=\BZ H$,
    \item[(ii)]  two integers $r, m$ with 
    \[
    2r \mid m^2H^2, \quad \mathrm{gcd}\left(r, \frac{m^2H^2}{2r}\right)=1,
    \]
    satisfying (\ref{VB}).
    \item[(iii)] a parallel transport $\rho:  H^2(S^{[n]}, \BQ)  \to H^2(X, \BQ)$ satisfying
\[
\rho \left(- \frac{mH}{r} + \frac{\delta}{2}\right) = \frac{\CB}{\ell} + \CA + \CD,
\]
where $\CA \in H^2(X, \BZ)$ and $\CD\in \mathrm{Pic}(X)_\BQ$,
\end{enumerate}
Then there is an $\alpha$-twisted vector bundle on $X$ of rank $n!r^n$.
\end{prop}

\begin{proof}
By (ii) we may set 
\[
s: = \frac{m^2H^2}{2r} \in \BZ_{>0},
\]
which satisfies $\mathrm{gcd}(r,s) = 1$. 
The Mukai vector $v_0= (r, mH, s)$ is primitive and isotropic satisfying (\ref{VB}), from which we can construct the $K3$ surface $M$ as a fine moduli space, the vector bundle $\CU$ on $M\times S$, and ultimately the vector bundle $\CU^{[n]}$ on $M^{[n]}\times S^{[n]}$. Then the discussion before Proposition \ref{prop1.2} yields a twisted vector bundle $\CE_y$ of the desired rank with respect to the Brauer class
\[
\left[-\frac{\alpha_X}{n!r^n} \right] = \left[-\frac{\rho(\alpha_{S^{[n]}})}{n!r^n} \right] = \left[\frac{\CB}{\ell} +\CA +D \right] = \left[\frac{\CB}{\ell} \right] \in \left( \frac{H^2(X, \BZ)} {\mathrm{Pic}(X)}\right) \otimes \BQ/\BZ. \qedhere
\]
\end{proof}

\subsection{Removing  technical assumptions}

The coprime condition in Proposition \ref{prop1.2}(ii) is to guarantee that $M$ is a fine moduli space over $S$. The following result is a strengthened version of Proposition \ref{prop1.2}, which proves that this coprime condition, as well as the primitivity assumption of $v_0$, can be completely removed. This is more convenient in practice.

\begin{prop}\label{prop1.3}
   Assume all the assumptions of Proposition \ref{prop1.2} are satisfied, except the coprime assumption 
    \begin{equation}\label{coprime}
    \mathrm{gcd}\left(r, \frac{m^2H^2}{2r}\right)=1.
    \end{equation}
    There is an $\alpha$-twisted vector bundle on $X$ whose rank divides $n!r^n$.
\end{prop}

\begin{proof} We proceed with the following two steps.
\medskip

{\bf \noindent Step 1.} We first show that, if we only remove the coprime assumption (\ref{coprime}) from Proposition~\ref{prop1.2}(ii) while keeping the Mukai vector $v_0 = \left(r, mH, \frac{m^2H^2}{2r}\right)$ primitive, then there exists an $\alpha$-twisted vector bundle of rank $n!r^n$.

More precisely, assume that we have a $K3$ surface $S$ and two integers $r,m$ without (\ref{coprime}). The strategy is that we can always achieve our goal by moving via a parallel transport to a different $K3$ surface $S'$ which satisfies the coprime condition. The construction is as follows.

Set $s:= \frac{m^2H^2}{2r}$. We first claim that we can pick $k\in \BZ_{>0}$ such that
\begin{equation}\label{coprime1}
\mathrm{gcd}(r, s+km)=1.
\end{equation}
Indeed, since $v_0=(r, mH, s)$ is primitive, we must have
\[
\mathrm{gcd}\left(r, \mathrm{gcd}(s,m)\right)=1.
\]
Therefore
\[
\mathrm{gcd}(r, s+km) = \mathrm{gcd}\left(r, \frac{s}{\mathrm{gcd}(s,m)}+k\frac{m}{\mathrm{gcd}(s,m)}\right). 
\]
Since $\frac{s}{\mathrm{gcd}(s,m)}, \frac{m}{\mathrm{gcd}(s,m)}$ are coprime, a desired $k$ clearly exists.

Now, we pick $L\in H^2(S,\BZ)$ with 
\[
L\cdot H=1, \quad L^2 >0.
\]
We set
\[
s':= \frac{(mH+rkL)^2}{2r}.  
\]
Since $m^2H^2 = 2rs$, a direct calculation yields
\[
s' = s + km + rk^2\cdot \frac{L^2}{2} \in \BZ_{>0},
\]
and consequently
\[
\mathrm{gcd}(r, s') = \mathrm{gcd}(r, s+km) =1
\]
where we used (\ref{coprime1}) in the last equation.

Finally, we find a projective $K3$ surface $S'$ of Picard rank $1$ and a parallel transport 
\[
\varphi: H^2(S, \BZ) \to H^2(S', \BZ)
\]
such that the class $\varphi(mH+rkL) \in H^2(S', \BZ)$ is algebraic; the existence of $S'$ is guaranteed by the global Torelli theorem. The new $K3$ surface $S'$ and the Mukai vector
\[
v_0':= (r, \varphi(mH+rkL), s') \in H^*(S', \BZ)
\]
satisfies 
\[
{v'_0}^2 = 0, \quad \mathrm{gcd}(r,s')=1,
\]
and (\ref{VB}), to which we can apply Proposition \ref{prop1.2}. This completes Step 1.
\medskip

{\noindent \bf Step 2.} In general, for a non-necessarily primitive Mukai vector
\[
v_0 = \left(r, mH, \frac{m^2H^2}{2r}\right),
\]
we can write
\[
v_0 = a v_0'
\]
with $a\in\BZ_{>0}$ and $v_0'$ primitive. Then applying Step 1 to $v'_0$ shows that there exists an $\alpha$-twisted vector bundle with rank $n! \left(\frac{r}{a}\right)^n$, which clearly divides $n!r^n$ as desired.

We have completed the proof of the proposition.
\end{proof}

In the next two sections, we explore how to use Markman's projectively hyperholomorphic bundle, via Proposition \ref{prop1.3}, to give new effective bounds for the period--index problem. This question is closely related to the arithmetic of the BBF form on $\mathrm{Pic}(X)$.

\section{Proofs of Theorems \ref{thm0.2} and \ref{thm0.4}}

In this section, we still assume that $X$ is a hyper-K\"ahler variety of $K3^{[n]}$-type with $n \geq 2$. Our purpose is to prove Theorems \ref{thm0.2} and \ref{thm0.4}. Since for both theorems we assume that $\mathrm{per}(\alpha)$ is coprime to $I_X$, the period $\mathrm{per}(\alpha)$ is always $\ell$.

\subsection{Lattice theory}
The second cohomology of $X$, endowed with the BBF form $q(-, -)$, is given by
\begin{equation}\label{lattice}
H^2(X, \BZ) \simeq U^{\oplus 3} \oplus E_8(-1)^{\oplus 2} \oplus \langle 2-2n \rangle,
\end{equation}
where $U$ stands for the rank $2$ hyperbolic lattice. The \emph{divisibility} of a class $L\in H^2(X,\BZ)$ is the positive generator $\mathrm{div}(L)$ of the subgroup
\[
\left\{ q(L, L') \mid L'\in H^2(X, \BZ)  \right\} \subset \BZ.
\]
We recall the following obvious fact.
\begin{lem}\label{lem2.1}
    The divisibility of any primitive class $L \in H^2(X, \BZ)$ satisfies 
\[
\mathrm{div}(L) \mid 2n-2.
\]
\end{lem}
\begin{proof}
This follows directly from the definition, combined with the fact that the first two summands of the right-hand side of (\ref{lattice}) form a unimodular lattice.
\end{proof}

For a given class $L\in H^2(X,\BZ)$, we use $L^\perp$ to denote the sub-lattice of $H^2(X,\BZ)$ given by the orthogonal complement of $L$ with respect to the BBF form $q(-, -)$.

\begin{lem}\label{lem2.2}
Assume that $\CB,\CD$ are two primitive classes in $H^2(X, \BZ)$ with $\CB$ lying in the transcendental part $T(X) \subset H^2(X, \BZ)$ and $\CD \in \mathrm{Pic}(X)$. There exists a class $\CA$ lying in the sub-lattice $\CB^\perp \cap \CD^\perp \subset H^2(X, \BZ)$, which is of divisibility $1$ in the lattice $\CB^\perp \cap \CD^\perp$. 
\end{lem} 
 
 \begin{proof}
We fix an isomorphism of lattices
\[
H^2(X, \BZ) \simeq U^{\oplus 3} \oplus E_8(-1)^{\oplus 2} \oplus \BZ\delta, \quad \delta^2 = 2-2n,
\]
where the right-hand side is naturally embedded into the unimodular lattice 
\[
 \Lambda: = U^{\oplus 4} \oplus E_8(-1)^{\oplus 2},
\]
known as the Markman--Mukai lattice. We consider the saturation of the sub-lattice generated by $\CB, \CD, \delta$ with the canonical primitive embedding:
\[
\iota: \Lambda':= \langle \CB, \CD, \delta \rangle_{\mathrm{sat}} \hookrightarrow \Lambda. 
\]
On the other hand, we know from \cite[Proposition 1.8]{HK3} that there is a primitive embedding of abstract lattices:
\[
\Lambda' \hookrightarrow U^{\oplus 3}.
\]
Combined with the natural embedding 
\[
U^{\oplus 3} \oplus U' \hookrightarrow \Lambda
\]
given by the definition of $\Lambda$, we obtain 
\[
\iota': \Lambda' \hookrightarrow U^{\oplus 3} \xrightarrow{~(\mathrm{id},0)~} U^{\oplus 3} \oplus U' \hookrightarrow \Lambda.
\]
Then the uniqueness part of \cite[Corollary 1.9]{HK3} implies that $\iota,\iota'$ differ by an automorphism of~$\Lambda$. In other words, there exists an isomorphism
\[
\Lambda = U^{\oplus 3} \oplus U' \oplus E_8(-1)^{\oplus 2}
\]
such that the image of $\iota: \Lambda' \hookrightarrow \Lambda$ lies completely in the first factor $U^{\oplus 3}$. A primitive vector  
\[
\CA \in U' \cap H^2(X, \BZ) \subset \Lambda
\]
satisfies the desired properties.
\end{proof}

\subsection{Proof of Theorem \ref{thm0.2}}
Let $\CD \in \mathrm{Pic}(X)$ be a primitive polarization of degree
\[
q(\CD) = 2h >0.
\]
We fix a Brauer class
\[
\alpha = \left[\frac{\CB}{\ell}\right] \in \mathrm{Br}(X), \quad \CB \in T(X), \quad \mathrm{per}(\alpha) = \ell;
\]
we further assume that
\begin{equation}\label{assumption0}
\mathrm{gcd}(\ell, n!h) = 1.
\end{equation}
Our goal is to construct an $\alpha$-twisted vector bundle on $X$ whose rank divides $\left(2^{2n}n!\right)\cdot \ell^{2n}$. Then by Lemma \ref{lem1.1} we have
\[
\mathrm{ind}(\alpha ) \mid \mathrm{per}(\alpha)^{2n}
\]
which implies Theorem \ref{thm0.2}.

We choose $\CA \in H^2(X, \BZ)$ as in Lemma \ref{lem2.2} which we fix from now on. Define
\[
L_u: = 2\ell  \CA + 2\CB + u\CD \in H^2(X, \BZ), \quad  u \in \BZ_{>0}. 
\]

\begin{prop}\label{prop2.3}
There exist odd integers $u_1,u_2$ with \[
\ell \mid u_1, \quad \mathrm{gcd}(u_2,\ell)=1
\]
satisfying that for both $i=1,2$:
\begin{enumerate}
    \item[(i)] $L_{u_i}$ is primitive,
    \item[(ii)] $q(L_{u_i}) >0$,
    \item[(iii)] $\mathrm{div}(L_{u_i}) =1$ or $2$.
\end{enumerate}
\end{prop}

\begin{proof}
We complete the proof via two steps.

\medskip
{\bf \noindent Step 1: finding $u_1$.}
Let $u_1$ be a sufficiently large odd multiple of $\ell$, so that (ii) holds. We check in the following that the conditions (i, iii) are satisfied for $L_{u_1}$.

Note that $\frac{L_{u_1}}{2}$ is not integral since $u_1$ is odd. For (i) it suffices to show that $\frac{L_{u_1}}{p}$ cannot be integral for any odd prime number $p$.

Assume that $\frac{L_{u_1}}{p} \in H^2(X, \BZ)$ with $p$ an odd prime number. Then by the definition of $\CA$, there exists $L \in H^2(X, \BZ)$ such that 
\begin{equation}\label{LLL}
q(L, \CA) =1, \quad q(L, \CB) =q(L, \CD)=0.
\end{equation}
Therefore we have
\[
q\left(L, \frac{L_{u_1}}{p}\right) = \frac{2\ell}{p} \in \BZ,
\]
which implies 
\begin{equation}\label{eqn8}
p \mid \ell.
\end{equation}    

Now, if we have $\ell \mid u_1$, then it cannot happen that
\[
\frac{L_{u_1}}{p} \in H^2(X, \BZ).
\]
Otherwise, since $p$ is a prime factor of $\ell$ by (\ref{eqn8}), this implies 
\[
\frac{2\CB}{p} \in H^2(X,\BZ),
\]
contradicting the assumption that $\mathrm{per}(\alpha) = \ell$.

We then observe that (iii) holds automatically for all the $L_{u_1}$ satisfying (i, ii).

Let $L$ be a class satisfying (\ref{LLL}). We have
\[
q(L, L_{u_1}) = 2\ell.
\]
Consequently 
\[
\mathrm{div}(L_{u_1}) \mid 2\ell.
\]
On the other hand, by (i) and Lemma \ref{lem2.1}, we obtain
\[
\mathrm{div}(L_{u_1}) \mid 2n-2.
\]
Then the proposition follows from 
\[
\mathrm{gcd}(2\ell, 2n-2) = 2
\]
which is given by the assumption (\ref{assumption0}).

\medskip
{\bf \noindent Step 2: finding $u_2$.}
Let $u_2$ be a sufficiently large odd positive integer coprime to $\ell$. By the observation above, it suffices to check that $L_{u_2}$ is primitive.

If $L_{u_2}$ is divisible by an odd prime number $p$, we deduce from (\ref{eqn8}) that $p$ is a factor of $\ell$. On the other hand, we know that 
\[
p \mid q(L_{u_2}, \CD) = u_2q(\CD).
\]
By our assumptions on $\ell, u_2$, we have 
\[
\mathrm{gcd}(\ell, u_2q(\CD)) =\mathrm{gcd}(\ell, 2u_2h) = 1.
\]
This is a contradiction; hence $L_{u_2}$ has to be primitive.
\end{proof}

We are now ready to prove Theorem \ref{thm0.2}; we separate the two cases:

\begin{enumerate}
    \item[(A)] $\ell \nmid 2q(\CB) +1$;
    \item[(B)] $\ell \mid 2q(\CB)+1$.
\end{enumerate}

\begin{proof}[Proof in the case (A)]
We pick $u_1$ given by Proposition \ref{prop2.3} and consider the class
\[
L_{u_1} \in H^2(X, \BZ).
\]
Since for any $u$ we have
\[
\frac{1}{2}q(L_{u}) \equiv 2q(\CB) + \frac{1}{2}u^2q(\CD) \mod \ell,
\]
our choice of $u_1$ and the assumption (A) ensure that
\begin{equation}\label{VB1}
\ell \nmid \frac{1}{2}q(L_{u_1}) +1.
\end{equation}
In the following, this equation will be used to verify the condition (\ref{VB}) for the Mukai vectors we construct.
\medskip

{\bf \noindent Case (A.1): $\mathrm{div}(L_{u_1})=1$.} By the Eichler criterion (see~\cite[Lemma 3.5 and Example 3.8]{GHS}) and \cite[Theorem 9.8]{Torelli}, there exist a polarized $K3$ surface $(S,H)$ with
\[
\mathrm{Pic}(S) = \BZ H, \quad H^2= (2n-2)\ell^2+ q(L_{u_1}),
\]
and a parallel transport $\rho: H^2(S^{[n]}, \BZ) \to H^2(X,\BZ)$ satisfying 
\[
\rho(H-\ell \delta) = \epsilon L_{u_1} , \quad \epsilon  = 1 \textrm{ or } {-1}.
\]
By setting 
\[
r: = 4\ell^2, \quad  m := -2\ell \epsilon,
\]
we can check directly that $(S,H, r, m)$ satisfies the assumptions of Proposition \ref{prop1.3}; here (\ref{VB}) is given by (\ref{VB1}):
\[
\frac{1}{2}H^2+1 \equiv \frac{1}{2}q(L_{u_1}) +1 \mod \ell  \quad  \implies \quad \ell \nmid \frac{1}{2}H^2+1.
\]
Hence there exists an $\alpha$-twisted vector bundle of rank $n!r^n = \left(2^{2n}n!\right)\cdot \ell^{2n}$ as desired.
\medskip

{\bf \noindent Case (A.2): $\mathrm{div}(L_{u_1})=2$.} Note that a class of divisibility 2 automatically satisfies
\[
q(L_{u_1}) \equiv 2-2n \mod 8.
\]
As in the first case, there exist a polarized $K3$ surface $(S,H)$ with
\[
\mathrm{Pic}(S) = \BZ H, \quad H^2= \frac{1}{4}\left((2n-2)\ell^2+ q(L_{u_1})\right),
\]
and a parallel transport $\rho: H^2(S^{[n]}, \BZ) \to H^2(X,\BZ)$ satisfying 
\[
\rho(2H-\ell \delta) = \epsilon L_{u_1} , \quad \epsilon  = 1 \textrm{ or } {-1}.
\]
Setting 
\[
r: = 4\ell^2, \quad  m := -4\ell \epsilon,
\]
the quadruple $(S,H, r, m)$ satisfies the assumptions of Proposition \ref{prop1.3}; similarly we need (\ref{VB1}) to guarantee (\ref{VB}). This constructs an $\alpha$-twisted vector bundle of rank $n!r^n = \left(2^{2n}n!\right)\cdot \ell^{2n}$ as desired.
\end{proof}

\begin{proof}[Proof in the case (B)] The proof in the case (B) is parallel; the invariants $r,m$ are constructed by the same formulas as in the case (A). The only difference is that we replace $u_1$ by $u_2$ of Proposition \ref{prop2.3}. More precisely, we pick $u_2$ given by Proposition \ref{prop2.3} and consider the class
\[
L_{u_2} \in H^2(X, \BZ).
\]
Note that
\[
\frac{1}{2}q(L_{u_2}) +1 \equiv (2q(\CB) +1) +\frac{1}{2} u_2^2q(\CD) \equiv \frac{1}{2} u_2^2q(\CD) \mod \ell
\]
where the second equation is given by the assumption (B). By the assumptions of $u_2,\ell$, we have 
\begin{equation}\label{VB2}
    \ell \nmid \frac{1}{2}q(L_{u_2}) +1.
\end{equation}
The rest of the proof is identical to that for the case (A), where (\ref{VB2}) plays the role of (\ref{VB1}).
\end{proof}

\subsection{Numerical obstructions}\label{NumOb}
In the proof above, we applied Proposition \ref{prop1.3} with $r =  4\ell^2$ which gives the bound (\ref{weak_bound}). A natural question is that, can we use this method to achieve the stronger bound (\ref{H_Conj})?
In the following, we discuss numerical obstructions that explain why we do not expect Huybrechts' conjecture to be derived from Proposition \ref{prop1.3}.

If we want to get (\ref{H_Conj}) for a Brauer class $\alpha = \left[\frac{\CB}{\ell} \right]$ with $\ell$ a prime number, we need to set 
\[
r =  c\ell, \quad \mathrm{gcd}(c,\ell)=1.
\]
The condition (iii) of Proposition \ref{prop1.2} (or Proposition \ref{prop1.3}) implies that there exists a parallel transport
\[
\rho: H^2(S^{[n]}, \BZ) \to H^2(X,\BZ)
\]
satisfying 
\[
-2m \rho(H) -c\ell \rho(\delta) = 2c \CB +2c\ell \CA +\CD
\]
with $\CA \in H^2(X,\BZ)$ and $\CD\in \mathrm{Pic}(X)$. The condition 
\[
\ell \mid m^2H^2
\]
(required by Proposition \ref{prop1.2} (ii)) combined with the fact that $\rho$ is an isometry further implies
\begin{equation}\label{main_obs}
\ell \mid q(\CD) + (2c)^2q(\CB).
\end{equation}

If $X$ has Picard rank $1$, it is impossible for (\ref{main_obs}) to hold for all but finitely many prime numbers $\ell$. To see this, we assume $\mathrm{Pic}(X) = \BZ \CR$, and then (\ref{main_obs}) is equivalent to finding an integer $k$ such that
\[
k^2 q(\CR) + (2c)^2 q(\CB) \equiv 0 \mod \ell.
\]
Consequently, this implies
\[
\left( kq(\CR) \right)^2 \equiv -(2c)^2q(\CB)q(\CR) \mod \ell,
\]
forcing the integer $-q(\CR)q(\CB)$ to be a quadratic residue modulo $\ell$. This cannot always hold, even allowing $\ell$ to avoid finitely many primes.

We now prove Theorem \ref{thm0.4} which shows that (\ref{main_obs}) is the \emph{only} obstruction.

\subsection{Proof of Theorem \ref{thm0.4}}

As before, we write the Brauer class $\alpha$ as
\[
\alpha = \left[\frac{\CB}{\ell}\right] \in \mathrm{Br}(X), \quad \CB \in T(X), \quad \mathrm{per}(\alpha) = \ell.
\]
Let $\CD \in \mathrm{Pic}(X)$ be a class satisfying  \begin{equation}\label{quadratic}
q(\CD) \equiv -q(\CB)d^2 \mod \ell, \quad \mathrm{gcd}(d,\ell) =1.
\end{equation}
Recall the class
\[
L_u = 2\ell  \CA + 2\CB + u\CD \in H^2(X, \BZ).
\]
We may assume that $q(\CD) >0$ and $\frac{\CD}{2}$ is not integral; otherwise we modify it by a multiple of~$\ell$ times an ample class. We pick an odd $u \in \BZ_{>0}$ such that
\begin{equation}\label{old_i_ii}
du \equiv 2 \mod \ell, \quad \quad  \quad q(L_u)>0,
\end{equation}
and fix it from now on. 

We note the following lemma; it is the \emph{only} place where (\ref{quadratic}) is essentially used.

\begin{lem}\label{lem2.4}
With $u$ as chosen above, we have 
    \[
    \ell \mid q(L_u).
    \]
\end{lem}

\begin{proof}
A direct calculation yields
\[
q(L_u) \equiv 4q(\CB) + u^2q(\CD) \mod \ell.
\]
Combining with (\ref{quadratic}) and the assumption (\ref{old_i_ii}) above, the right-hand side is
\[
q(\CB)\left( 4- (du)^2 \right) \equiv 0 \mod \ell. \qedhere
\]
\end{proof}

Next, we show that the assumption (ii) ensures that the class $L_u$ is primitive.

\begin{lem}
    The class $L_u$ is primitive, and the divisibility of $L_u$ is either $1$ or $2$.
\end{lem}

\begin{proof}
We write
\[
L_u = a L_u'
\]
with $a\in \BZ_{>0}$ an odd integer and $L_u'$ primitive. The same argument as in Proposition \ref{prop2.3} shows 
\[
a\cdot\mathrm{div}(L'_u)= \mathrm{div}(L_u) \mid 2\ell, \quad \mathrm{div}(L'_u) \mid 2n-2,
\]
which further implies 
\[
\mathrm{div}(L'_u) = 1 \textrm{ or } 2, \quad a \mid \ell.
\]
Recall the class $\CD'$ which satisfies the assumptiom (ii) in Theorem \ref{thm0.4}. We must have
\[
a \mid q(L_u, \CD') = u q(\CD, \CD');
\]
combined with 
\[
a \mid q(L_u, \CB) = 4q(\CB),
\]
we conclude that
\[
a \mid \mathrm{gcd}\left(q(\CD, \CD'), q(\CB), \ell \right) =1 ~~~\implies~~~ a=1. \qedhere
\]
\end{proof}

The remainder of the proof is parallel to that of Theorem \ref{thm0.2}; we note that in both cases 1 and 2 below, the condition (\ref{VB}) is satisfied automatically due to Lemma \ref{lem2.4}.
\medskip

{\bf \noindent Case 1: $\mathrm{div}(L_u)=1$.} There exist a polarized $K3$ surface $(S,H)$ with
\[
\mathrm{Pic}(S) = \BZ H, \quad H^2= (2n-2)\ell^2+ q(L_u),
\]
and a parallel transport $\rho: H^2(S^{[n]}, \BZ) \to H^2(X,\BZ)$ satisfying 
\[
\rho\left(H-\ell \delta\right) = \epsilon L_u , \quad \epsilon  = 1 \textrm{ or } {-1}.
\]
By setting 
\[
r: = 4\ell, \quad  m := -2 \epsilon,
\]
we can check directly that $(S,H, r, m)$ satisfies the assumptions of Proposition \ref{prop1.3}. Hence there exists an $\alpha$-twisted vector bundle whose rank divides $n!r^n = \left(2^{2n}n!\right)\cdot \ell^n$, and the theorem in this case follows from Lemma \ref{lem1.1}.

\medskip

{\bf \noindent Case 2: $\mathrm{div}(L_u)=2$.} There exist a polarized $K3$ surface $(S,H)$ with
\[
\mathrm{Pic}(S) = \BZ H, \quad 4H^2= (2n-2)\ell^2+ q(L_u),
\]
and a parallel transport $\rho: H^2(S^{[n]}, \BZ) \to H^2(X,\BZ)$ satisfying 
\[
\rho\left(2H-\ell \delta\right) = \epsilon L_u , \quad \epsilon  = 1 \textrm{ or } {-1}.
\]
By setting 
\[
r: = 4\ell, \quad  m := -4 \epsilon,
\]
we can check directly that $(S,H, r, m)$ satisfies the assumptions of Proposition \ref{prop1.3}. Hence there exists an $\alpha$-twisted vector bundle whose rank divides $n!r^n = \left(2^{2n}n!\right)\cdot \ell^n$ as desired.

The proof of Theorem \ref{thm0.4} is completed. \qed

\subsection{Theorem \ref{thm0.1} for $K3^{[n]}$-type revisited}\label{revisit}

We explain in this section that Theorem \ref{thm0.1} for $K3^{[n]}$-type varieties is an immediate consequence of Theorem \ref{thm0.4}.

Assume that $X$ is of $K3^{[n]}$-type admitting a Lagrangian fibration. It has Picard rank $\geq 2$, and the BBF form of the rank $2$ lattice spanned by a polarization and the pullback of the hyperplane class of the base is
\[
Q(x,y) = C_X\cdot xy,
\]
where $C_X \in \BZ$ is an intrinsic invariant of $X$. If $\ell$ is coprime to $C_X$, it is clear that 
\[
Q(x,y) \equiv -q(\CB)d^2 \mod \ell, \quad \mathrm{gcd}(d,\ell)=1
\]
has an integral solution. Moreover, there exist integral $x',y'$ such that
\[
\mathrm{gcd}(Q(x',y'), \ell) =1.
\]
Therefore Theorem \ref{thm0.1} for $X$ follows from Theorem \ref{thm0.4} with 
\[
N_X:= C_X\cdot n!I_X.
\]
This proves Theorem \ref{thm0.1} for $K3^{[n]}$-type varieties.    \qed

We have explained in (\ref{NumOb}) that (\ref{quadratic}) may not be solvable if $X$ has Picard rank $1$. In fact, except the case of Lagrangian fibrations as we discussed above, the equation (\ref{quadratic}) is not solvable for most rank $2$ Picard lattices.

To see a typical example, we consider the quadratic form
\[
Q(x,y) = 2x^2 -10y^2.
\]
If $\ell=9$ and $q(\CB) \equiv 3 \mod 9$, then (\ref{quadratic}) is not solvable. Indeed, a solution to 
\[
Q(x,y) \equiv 0 \mod 3
\]
must satisfy that 
\[
x\equiv y \equiv 0 \mod 3; 
\]
in this case $Q(x,y)$ has to be a multiple of $9$. We discuss this phenomenon geometrically in Section \ref{sec3.2}; in particular, for Picard rank $2$, this is essentially the \emph{only} case where (\ref{quadratic}) is not solvable, and the \emph{non-special} condition in Theorem \ref{thm0.6} rules it out.

\section{Proof of Theorem \ref{thm0.6}}

Let $n \geq 2$. We assume that $X$ is a hyper-K\"ahler variety of $K3^{[n]}$-type of Picard rank 
\[
\mathrm{rk}(\mathrm{Pic}(X)) = \mu \geq 1.
\]
Our purpose is to show that Theorem \ref{thm0.6} is a consequence of Theorem \ref{thm0.4}. Since we assume that the Brauer class $\alpha = [\frac{\CB}{\ell}]$ is non-special, the assumption (ii) of Theorem \ref{thm0.4} is always satisfied. 

As before, we further assume that $\mathrm{per}(\alpha)$ is coprime to $I_X$ so that $\mathrm{per}(\alpha) = \ell$. We study in this section the congruence equation
\begin{equation}\label{quadratic0}
q(\CD) \equiv -q(\CB)d^2 \mod \ell, \quad \mathrm{gcd}(d,\ell) =1,
\end{equation}
which completes the proof.

\subsection{Rational points}
We write the Brauer class $\alpha = \left[ \frac{\CB}{\ell}\right]$ as before. In this section, we reduce the existence of $\CD,d$ satisfying (\ref{quadratic0}) to the existence of a nonsingular rational point on a quadric hypersurface.

We use $Q$ to denote the BBF form on $\mathrm{Pic}(X)$ which is non-degenerate, and we set
\[
b:= -q(\CB) \in \BZ. 
\]
We consider 
\[
V:=\{ Q(x_1, \cdots, x_\mu)  - bw^2=0\} \subset \BP^\mu.
\]
Note that $V$ actually defines a scheme over $\BZ$, which further induces a variety $V_p$ over $\overline{\BF}_p$ for every prime $p$.

\begin{prop}\label{prop3.1}
Assume that for every prime factor $p$ of $\ell$ there is an $\BF_p$-rational point lying in the Zariski open subset
\[
V_p \cap \{w \neq 0\}  \subset V_p
\]
which is a nonsingular point of the $\overline{\BF}_p$-variety $V_p$. Then there exists $\CD \in \mathrm{Pic}(X)$ and $d$ satisfying (\ref{quadratic0}).
\end{prop}

\begin{proof}
    We proceed the proof by the following steps. We assume that 
    \[
    \ell = p_1^{e_1}p_2^{e_2}\cdots p_t^{e_t}
    \]
    with $p_i$ pairwise distinct primes.
    
\medskip
{\bf \noindent Step 1.} Our purpose is to find integers $x_i$ and integer $w$ with $\mathrm{gcd}(w,\ell)=1$ such that
\begin{equation}\label{original}
Q(x_1, \cdots, x_\mu)  - bw^2 \equiv 0 \mod \ell.
\end{equation}
These integers satisfy automatically
\begin{equation}\label{mod_p_power}
Q(x_1, \cdots, x_\mu)  - bw^2 \equiv 0 \mod p_j^{e_j}, \quad \mathrm{gcd}(w,p_j)=1
\end{equation}
for every $i \in \{1,2,\cdots,t\}$. Conversely, if we can find solutions 
\[
x_{i,1}, x_{i,2}, \cdots, x_{i,\mu}, w_i \in \BZ
\]
to (\ref{mod_p_power}) for every $i$, then the Chinese Remainder Theorem implies that there exists a solution to the congruence equations
\[
x_j \equiv x_{i,j} \mod  p_j^{e_j}, \quad w \equiv w_j \mod p_j^{e_j},
\]
which further gives a solution to (\ref{original}).

\medskip

{\bf \noindent Step 2.} By Step 1, from now on we only focus on a prime $p\in \{p_1, p_2, \cdots, p_t\}$ and consider the equation
\begin{equation}\label{mod_mod_p}
Q(x_1, \cdots, x_\mu)  - bw^2\equiv 0 \mod p^{e}, \quad \mathrm{gcd}(w,p)=1.
\end{equation}

If $V_p \cap \{w \neq 0\}$ has an $\BF_p$-rational point, we may assume that there are integers $x_i,w$ satisfying 
\[
Q(x_1, \cdots, x_\mu)  - bw^2 \equiv 0 \mod p, \quad \mathrm{gcd}(w,p)=1.
\]
Furthermore, if this rational point is nonsingular on $V_p$, then by Hensel's lemma this $\BF_p$-point can be lifted to a $\BZ_p$-point. In particular, we can find a solution to the congruence equation~(\ref{mod_mod_p}).
\end{proof}

\subsection{Proof of Theorem \ref{thm0.6}}\label{sec3.2}

We first revisit the example in Section \ref{revisit} from the perspective of Proposition \ref{prop3.1}; recall that
\[
Q(x, y)= 2x^2 -10y^2, \quad p=3, \quad \ell = p^2 = 9, \quad q(\CB)=12. 
\]
The variety $V_3$ is given by
\[
V_3 =\{ 2x^2-10y^2 = 0 \} \subset \BP^2_{x,y,z}
\]
where $[0:0:1]$ is the only rational point. Geometrically, the variety $V_3$ is the union of $2$ lines, but both lines are not defined over $\BF_3$. The intersection of the two lines is the unique rational point, which is singular. Note that this is essentially the only bad situation: the conic is degenerate and the only rational point is the node. 

Now we prove Theorem \ref{thm0.6}.

To illustrate the idea, we first treat the $\mu=2$ case.  Since $Q(x,y)$ is non-degenerate, there are only finitely many primes $p$ satisfying that the affine curve
\[
\{ Q(x,y) = 0\}\setminus \{(0,0)\} \subset \BA^2
\]
is singular over $\overline{\BF}_p$. Let $C_Q$ be the product of all these primes. Then for any prime $p$ coprime to $C_Q$ and any $b$ coprime to $p$, the projective curve $V_p$ is nonsingular. So it is a nonsingular conic, and is isomorphic to $\BP^1$. Hence its affine part $\{z\neq 0\}$ must contain a rational point. By Proposition \ref{prop3.1}, the theorem holds with the constant $N_X := C_Q\cdot n!$.

The general case is similar. Consider $C_Q$ given by the product of all primes $p$ satisfying that the affine cone 
\[
\{Q(x_1,x_2,\cdots, x_\mu) = 0 \} \subset \BA^\mu
\]
has singular point other than $(0,0,\cdots, 0)$ over $\overline{\BF}_{p}$, and again set
\[
N_X:= C_Q \cdot n!I_X.
\]
To complete the proof, we show that when $\mathrm{gcd}(p, N_X)=1$ the equation
\[
Q(x_1,x_2, \cdots, x_\mu) \equiv bw^2 \mod p^e, \quad \mathrm{gcd}(w, p)=1
\]
always has a solution. Since $\mathrm{gcd}(b,p)=1$, the variety $V_p$ is a nonsingular quadric hypersurface in $\BP^\mu$; therefore, the assumption of Proposition \ref{prop3.1} is satisfied, which implies the existence of desired $\CD,d$. \qed

\end{document}